\documentclass{amsart}

\usepackage{amsmath,amsthm,amssymb}

\newtheorem{prop}{Proposition}
\newtheorem{theorem}{Theorem}
\newtheorem{lemma}[prop]{Lemma}
\theoremstyle{definition}

\theoremstyle{remark}
\newtheorem{remark}[prop]{Remark}
\newtheorem{example}[prop]{Example}
\newtheorem*{conj}{Conjecture}

\newcommand{\inv}{^{-1}}
\newcommand{\sidehat}{^{\wedge}}
\newcommand{\unit}{^{(0)}}
\newcommand{\R}{\mathbb{R}}

\newcommand{\N}{\mathbb{N}}
\newcommand{\Z}{\mathbb{Z}}
\newcommand{\Q}{\mathbb{Q}}
\newcommand{\T}{\mathbb{T}}

\allowdisplaybreaks

\title{Groupoid $C^*$-algebras with Hausdorff Spectrum}
\author{Geoff Goehle}
\address{Mathematics and Computer Science Department, Stillwell 426,
  Western Carolina University, Cullowhee, NC 28723}
\email{grgoehle@email.wcu.edu}
\subjclass[2010]{22A22,47A67,46L99}

\begin{document}

\begin{abstract}
Suppose $G$ is a second countable, locally compact Hausdorff groupoid
with abelian stabilizer subgroups and a Haar system.  
We provide necessary and sufficient conditions for the groupoid
$C^*$-algebra to have Hausdorff spectrum.  In particular we show that
the spectrum of $C^*(G)$ is Hausdorff if and only if the stabilizers
vary continuously with respect to the Fell topology, the orbit space
$G\unit/G$ is Hausdorff, and, given convergent sequences $\chi_i\to
\chi$ and $\gamma_i\cdot\chi_i \to \omega$ in the dual stabilizer
groupoid $\widehat{S}$ where the $\gamma_i\in G$ act via conjugation,
if $\chi$ and $\omega$ are elements of the same fiber then $\chi =
\omega$.  
\end{abstract}

\maketitle


\section*{Introduction}

One of the reasons that $C^*$-algebras are so well studied is that
they have a very deep representation theory.  Understanding the
spectrum or primitive ideal space of a $C^*$-algebra, and in
particular the topology on these spaces, can reveal a great deal of
information about the underlying algebra.  For example, if a separable
$C^*$-algebra $A$ has Hausdorff spectrum $\widehat{A}$ then $A$ is
naturally isomorphic to the section algebra of an upper-semicontinuous
bundle over $\widehat{A}$ such that each fiber of the bundle is
isomorphic to the compact operators.  The continuous trace
$C^*$-algebras, which can be classified by a cohomology element,
are then algebras with Hausdorff spectrum whose
associated bundles are ``locally trivial'' in an appropriate sense
\cite[Chapter 5]{tfb}.  Given a class of $C^*$-algebras it is an
interesting problem to characterize those 
algebras which have Hausdorff spectrum. 

For example, in \cite{tghs} the author proves the following result.
Suppose we are given a transformation group $(H,X)$ such that $H$ is
abelian and the group action satisfies any of the conditions in the
Mackey-Glimm dichotomy \cite{groupoiddichotomy}.  Then the
transformation group $C^*$-algebra will have Hausdorff spectrum if and
only if the stabilizer subgroups of the action vary continuously with
respect to the Fell topology and the orbit space $X/H$ is Hausdorff.  
In this paper we would like to
extend the work of \cite{tghs} from transformation groups to
groupoids.  The most straightforward  generalization is the conjecture
that, given a groupoid $G$ with abelian stabilizer subgroups which satisfies
the conditions of the Mackey-Glimm dichotomy, the groupoid
$C^*$-algebra will have Hausdorff spectrum if and only if the
stabilizers vary continuously in $G$ and $G\unit/G$ is Hausdorff.
Interestingly, we will show that this ``naive'' generalization fails and
that characterizing the groupoid $C^*$-algebras with Hausdorff
spectrum requires a third condition.  Furthermore, the correct
generalization, presented in Section \ref{sec:groupoid-c-algebras} as 
Theorem \ref{thm:groupoidresult}, is in some
ways stronger than the results of \cite{tghs}, even for
transformation groups.  We finish the
paper by providing some further examples in Section
\ref{sec:hausd-spectr-dual}.  In addition, we also prove that, unlike the
$T_0$ or $T_1$ case, in the Hausdorff case the spectrum cannot be
studied using only the stabilizer subgroupoid.  

Before we get started we should review some preliminary
material.  Throughout the paper we will let $G$ denote a second
countable, locally compact Hausdorff groupoid with a Haar system
$\{\lambda_u\}$.  We will use $G\unit$ to denote the unit space, $r$ to
denote the range map, and $s$ to denote the source map.   We will let
$S = \{\gamma\in G : s(\gamma) = r(\gamma)\}$ be the stabilizer, or
isotropy, subgroupoid of $G$.  Observe that on $S$ the range and
source maps are equal and that $r=s:S\to G\unit$ gives $S$ a bundle
structure over $G\unit$.  Given $u\in G\unit$ the fiber $S_u =
r|_S\inv(u)$ is a group and is called the stabilizer
subgroup at $u$.    
Since $S$ is a closed subgroupoid of $G$, it is always second
countable, locally compact, and Hausdorff.  However, $S$ will have a
Haar system if and only if the stabilizers vary continuously.  That
is, if and only if the map $u\mapsto S_u$ is continuous with respect
to the Fell topology on closed subsets of $S$ \cite[Lemma
1.3]{renaultgcp}.  

One of the primary examples of groupoids are those built from
transformation groups.  If a second countable locally compact
Hausdorff group $H$ acts on a second countable locally compact
Hausdorff space $X$ then we can form the transformation groupoid
$H\ltimes X$ in the usual fashion.  The properties of the
transformation groupoid are closely tied to those of the group
action.  For instance, the orbit space $H\ltimes X\unit/H\ltimes X$ 
is homeomorphic to the orbit space of the action $X/H$.  Furthermore, 
the stabilizer groups $S_X$ of $H\ltimes X$  can be naturally
identified with the stabilizer subgroups $H_x$ of $H$ with respect to the
group action and the stabilizers will vary continuously in
$H\ltimes X$ if and only if they vary continuously in $H$.  

Given a groupoid $G$ we can construct the groupoid
$C^*$-algebra $C^*(G)$ as a universal completion of the convolution
algebra $C_c(G)$ \cite{groupoidapproach, coords}.  Of particular interest
to us will be the spectrum $C^*(G)\sidehat$ of the groupoid algebra.
One special case which will play a key role in our results 
is the spectrum of the
stabilizer subgroupoid.  Suppose that $G$ has abelian stabilizer
subgroups, that is, suppose the fibers of $S$ are all abelian.  If the
stabilizers vary continuously so that $S$ has a Haar system then 
we may construct the groupoid
algebra $C^*(S)$.  It turns out that in this case $C^*(S)$ is abelian
and the spectrum of $C^*(S)$, 
denoted by $\widehat{S}$, is a second countable locally
compact Hausdorff space which is naturally fibered over $G\unit$.
Furthermore the fiber of $\widehat{S}$ over $u\in G\unit$, which we
will write as
$\widehat{S}_u$, is the Pontryagin dual of the fiber $S_u$
\cite[Section 3]{ctgIII}.  We refer to $\widehat{S}$ as the dual
stabilizer groupoid.  One of the things that makes $\widehat{S}$ so
useful is that its topology is relatively well understood;
\cite{ctgIII} gives a complete description of the
convergent sequences in $\widehat{S}$.  Since we will use this
characterization quite a bit we have restated it below. 

\begin{prop}[{\cite[Proposition 3.3]{ctgIII}}]
\label{prop:3}
Suppose the groupoid $G$ has continuously varying abelian
stabilizers and that $\{\chi_n\}$ is a sequence in $\widehat{S}$ with
$\chi_n \in \widehat{S}_{u_n}$ for all $n$.  Given $\chi\in
\widehat{S}_u$ we have $\chi_n \to \chi$ if and only if 
\begin{enumerate}
\item $u_n \to u$ in $G\unit$, and 
\item given $s_n \in S_{u_n}$ for all $n$ and $s\in S_u$ if $s_n\to s$
  then $\chi_n(s_n)\to \chi(s)$.  
\end{enumerate}
\end{prop}

The final thing we need to review is the notion of a groupoid action.
A groupoid $G$ can only act on spaces $X$ which are fibered over $G\unit$.
If there is a surjective function $r_X:X\to G\unit$ then we define a groupoid
action via a map $\{(\gamma,x):s(\gamma)=r_X(x)\}\to X$ such that for
composable $\gamma$ and $\eta$ we have $\gamma\cdot(\eta\cdot x) =
\gamma\eta\cdot x$.  Among other things, this implies that $r_X(x)\cdot x
= x$ for all $x\in X$ and $r_X(\gamma\cdot x) = r(\gamma)$.  We will
use the following three actions in this paper.  Any
groupoid $G$ has actions on its unit space $G\unit$ and its
stabilizer subgroupoid $S$ which are defined as follows
\begin{align*}
\gamma\cdot u = \gamma u \gamma\inv = r(\gamma)\quad\text{on
  $G\unit$, and} \quad
\gamma\cdot s = \gamma s \gamma\inv \quad\text{on $S$.}
\end{align*}
Furthermore if $S$ has abelian fibers which vary
continuously then there is an action of $G$ on $\widehat{S}$.
For $\gamma\in G$, $\chi\in \widehat{S}_{s(\gamma)}$ we define
\[
\gamma\cdot \chi(s) = \chi(\gamma\inv s \gamma)\quad\text{for $s\in
  S_{r(\gamma)}$.}
\]
Given an action of $G$ on a space $X$ we will use $G\cdot x$ to denote
the orbit of $x$ in $X$ and $[x]$ to denote the corresponding element
of $X/G$.  We would also like to recall that the orbit space $X/G$ is
locally compact, but not necessarily Hausdorff, and that the quotient
map $q:X\to X/G$ is open as long as $G$ has a Haar system \cite[Lemma
2.1]{groupoidcohom}.   

\section{Groupoid $C^*$-algebras with Hausdorff Spectrum}
\label{sec:groupoid-c-algebras}

As mentioned in the introduction, we would like to generalize the main
result of \cite{tghs}, which has been restated below, from 
transformation groups to groupoids. 

\begin{theorem}[{\cite[Page 320]{tghs}}]
\label{thm:transresult}
Suppose that $(H,X)$ is an abelian transformation group and that the
maps of $H/H_x$ onto $H\cdot x$ are homeomorphisms for each $x\in X$.
Then the spectrum of the transformation group $C^*$-algebra $C^*(H,X)$
is Hausdorff if and only if the map $x\mapsto
H_x$ is continuous with respect to the Fell topology and $X/H$ is Hausdorff. 
\end{theorem}

\begin{remark} 
The condition that the maps of $H/H_x$ onto $H\cdot x$ are
homeomorphisms for each $x\in X$ is one of the equivalent conditions in the
Mackey-Glimm dichotomy \cite{groupoiddichotomy}.  
Following \cite{specpaper} we will refer to 
groupoids and transformation groups which satisfy one, and hence
all, of the conditions of the Mackey-Glimm dichotomy as {\em  regular}. 
\end{remark}

An important question is how to generalize the hypothesis that the group
$H$ is abelian.  The most natural replacement is to assume 
that the stabilizer subgroups $S_u$ are abelian for all $u\in G\unit$.
Since, as we will see, the regularity hypothesis can be removed
completely, this leaves us with the following conjecture. 

\begin{conj}
\label{conj-1}Suppose the groupoid $G$ has abelian stabilizers.  Then $C^*(G)$ will
have Hausdorff spectrum if and only if the stabilizers vary
continuously and $G\unit/G$ is Hausdorff. 
\end{conj}

However, we will find that this conjecture fails and the assumption that
$G$ has abelian stabilizers is a weaker condition, even for 
transformation groups.
Let us start by assuming $G$ is a second countable, locally
compact Hausdorff groupoid with abelian stabilizers and that
$C^*(G)\sidehat$ is Hausdorff.  It then follows from 
\cite[Proposition 3.1]{ctgIII} 
that the stabilizers must vary continuously.  Next consider the
following useful lemma.  

\begin{lemma}
\label{lem:1}
Suppose $G$ is a second countable locally compact Hausdorff
groupoid with continuously varying abelian stabilizers.  Then the
following are equivalent:
\begin{enumerate}
\item $C^*(G)$ has $T_0$ spectrum.   
\item $C^*(G)$ is GCR.
\item $G\unit/G$ is $T_0$.
\end{enumerate}
Furthermore, if any of these conditions hold then the map $[\gamma]\to
r(\gamma)$ from $G_u/S_u$ to $G\cdot u$ is a homeomorphism for all
$u\in G\unit$ and $G$ is regular.
\end{lemma}

\begin{proof}
The groupoid algebra is separable since $G$ is second countable.  In
this case the equivalence of the first two conditions follows from 
\cite[Theorem 6.8.7]{pedersenauto}.  
Since the stabilizers are abelian, and therefore
amenable and GCR, the equivalence of the second two conditions now
follows from the main result of \cite{ccrgca}.  Finally, if $G\unit/G$ is
$T_0$ then it follows from
\cite{groupoiddichotomy} that the map $[\gamma]\mapsto r(\gamma)$ 
from $G_u/S_u$ onto $G\cdot u$ is a homeomorphism for
all $u\in G\unit$ and hence $G$ is regular in the sense of
\cite{specpaper}. 
\end{proof}

Since we have assumed $C^*(G)\sidehat$ is Hausdorff, 
Lemma \ref{lem:1} implies that $G$ is regular.  We may now 
use \cite[Theorem 3.5]{specpaper} to conclude that
$C^*(G)\sidehat$ is homeomorphic to $\widehat{S}/G$.  A brief
argument shows that $G\unit/G$ is homeomorphic to its image in
$\widehat{S}/G$ equipped with the relative topology.  Thus $G\unit/G$ is
Hausdorff.  This demonstrates one direction of our conjecture.  
On the other hand, suppose that $G$ has continuously varying
abelian stabilizers and that $G\unit/G$ is Hausdorff.  Then $G\unit/G$
is certainly $T_0$ so that $G$ is regular.  It then follows from
\cite[Theorem 3.5]{specpaper} that $C^*(G)\sidehat$ is homeomorphic to
$\widehat{S}/G$.  So we will have proven our conjecture if we can show that
$\widehat{S}/G$ is Hausdorff.  What is more, setting aside the issue
of continuously varying stabilizers for the moment, we also have the
following suggestive proposition. 

\begin{prop}
\label{prop:1}
Suppose $G$ is a second countable, locally compact Hausdorff groupoid
with continuously varying abelian stabilizers.  Then $C^*(G)\sidehat$
is $T_1$ (resp. $T_0$) if and only if $G\unit/G$ is $T_1$
(resp. $T_0$).
\end{prop}

\begin{proof}
It follows from Lemma \ref{lem:1} that $C^*(G)\sidehat$ is $T_0$ if
and only if $G\unit/G$ is.  Now suppose $C^*(G)\sidehat$ is $T_1$.
Then Lemma \ref{lem:1} and \cite[Theorem
3.5]{specpaper} imply that $C^*(G)\sidehat$ is homeomorphic to
$\widehat{S}/G$.  As noted above, 
$G\unit/G$ is homeomorphic to its image in $\widehat{S}/G$, and as
such $G\unit/G$ is $T_1$.  Next suppose that
$G\unit/G$ is $T_1$.  Again using Lemma \ref{lem:1} and \cite[Theorem
3.5]{specpaper} we have $C^*(G)\sidehat \cong \widehat{S}/G$.  Thus,
mirroring the Hausdorff case, we will be done if we can show that
$\widehat{S}/G$ is $T_1$.  

Suppose that we are given elements $[\rho],[\chi]\in \widehat{S}/G$
such that $[\rho]\ne[\chi]$.   Let $p:\widehat{S}\to G\unit$ be the
bundle map and $\tilde{p}:\widehat{S}/G\to G\unit/G$ its
factorization.  Set $[u] = \tilde{p}([\rho])$ and $[v] =
\tilde{p}([\chi])$.  Suppose $[u]\ne [v]$.  Since $G\unit/G$ is $T_1$
we can find open sets $U$ and $V$ such that $[u]\in U$, $[v]\in V$ and
$[u]\not\in V$, $[v]\not\in U$.  Then $\tilde{p}\inv(U)$ is an open
set containing $[\rho]$ and not $[\chi]$ and $\tilde{p}\inv(V)$ is an
open set containing $[\chi]$ and not $[\rho]$.  Next suppose $[u] =
[v]$.  Since the fibers of $S$ are abelian we have 
\begin{equation}
\label{eq:2}
s \cdot \chi(t) = \chi(s\inv t s) = \chi(t)\quad\text{for all $s\in S$.}
\end{equation}
Hence the action of $G$ on $S$ is trivial when fixed to a single fiber
and we can assume without loss of generality that $\rho,\chi\in
\widehat{S}_u$ with $\rho \ne \chi$.  Let $q:\widehat{S}\to
\widehat{S}/G$ be the quotient map and recall that it is open. Fix a
neighborhood $U$ of $\rho$.  If $\chi\not\in G\cdot U$ then $[\chi]\not
\in q(U)$ and $q(U)$ separates $[\rho]$ from $[\chi]$.  Now suppose
$\chi\in G\cdot U$ for all neighborhoods $U$ of $\rho$.  Then for each $U$
there exists $\gamma_U\in G$ and $\rho_U \in U$ such that $\rho_U =
\gamma_U\cdot \chi$.  If we direct $\rho_U$ by decreasing $U$ then it
is clear that $\rho_U \to \rho$.  This implies that 
\(
\gamma_U \cdot u = r(\gamma_U) = p(\rho_U) \to u. 
\)
Since $G$ is regular $[\gamma]\mapsto r(\gamma)$ is a homeomorphism
and we must have
$[\gamma_U] \to [u]$ in $G_u/S_u$.  However, the quotient map on
$G_u/S_u$ is open so that we may pass to a subnet, relabel, and choose
$r_U\in S_u$ such that $\gamma_Ur_U \to u$.  Using \eqref{eq:2}
\[
\gamma_Ur_U\cdot\chi = \gamma_U\cdot \chi = \rho_U \to u\cdot \chi =
\chi.
\]
Thus $\rho = \chi$, which is a
contradiction.  It follows that we must have been able to separate
$[\rho]$ from $[\chi]$.  This argument is completely symmetric so that
we can also find an open set around $[\chi]$ which does not contain
$[\rho]$.  It follows that $\widehat{S}/G$, and hence
$C^*(G)\sidehat$, is $T_1$. 
\end{proof}  

The essential component of this proof is the argument 
that $\widehat{S}/G$ is $T_1$
if $G\unit/G$ is $T_1$.  We would like to extend this to the Hausdorff
case but there are
topological obstructions.  We start by recalling Green's famous
example of a free group action that is not proper.

\begin{example}[\cite{tgasos}]
\label{ex:1}
The space $X\subset\R^3$ will consist of countably many orbits, with
the points $x_0=(0,0,0)$ and $x_n = (2^{-2n},0,0)$ for $n\in\N$
as a family of representatives.  The action of $\R$ on $X$ is
described by defining maps $\phi_n :\R\rightarrow X$ such that
$\phi_n(s) = s\cdot x_n$.  In particular we let
\(
\phi_0(s) = (0,s,0) 
\)
and for $n\geq 1$
\[
\phi_n(s) = \begin{cases} (2^{-2n},s,0) & s \leq n \\
(2^{-2n}-(s-n)2^{-2n-1}, n\cos(\pi(s-n)),n\sin(\pi(s-n))) & n < s <
n+1 \\
(2^{-2n-1}, s-1-2n,0) & s \geq n+1.
\end{cases}
\]
For instance, brief computations show that 
\begin{equation}
\label{eq:3}
2n+1\cdot (2^{-2n},0,0) = (2^{-2n-1},0,0)
\end{equation}
for all $n$.   It is straightforward to observe that
the orbit space $X/\R$ is homeomorphic to the subset
$\{x_n\}_{n=0}^\infty$ of $\R^3$.  
\end{example}

In the following we build an example of a transformation groupoid $G$
with continuously varying abelian stabilizers such that $G\unit/G$ is
Hausdorff and $\widehat{S}/G$ is not.  This shows that, even in the
transformation group case,  our conjecture
fails and that we cannot use the straightforward 
generalization of Theorem \ref{thm:transresult}.  

\begin{example}
\label{ex:3}
Let $\R$ act on $X$ as in Example \ref{ex:1}.  Now restrict this
action to the action of $\Z$ on the subset 
\(
Y= \{\phi_n(m):n\in \N, m\in \Z\}.
\)
Let $H = \Q_D\rtimes_\phi\Z$ be
the semidirect product, where $\Q_D$ denotes the rationals equipped
with the discrete topology and where we define 
\begin{equation}
\phi(n)(r) = r2^n
\end{equation}
for all $n\in \Z$ and $r\in \Q$.  It is easy to show that $\phi$
is a homomorphism from $\Z$ into the automorphism group of $\Q_D$.
Thus $H$ is a locally compact Hausdorff group which is second
countable because it is a countable discrete space.  Recall that the
group operations are given by 
\begin{align*}
(q,n)(p,m) & = (q + 2^n p, n + m) & (q,n)\inv &= (-2^{-n}q,-n).
\end{align*}
Let the second factor of
$H$ act on $Y$ as in Example \ref{ex:1}.  In other words, let
$(q,n)\cdot x := n \cdot x$.  It is straightforward to show that this
is a continuous group action.  It follows that the transformation
groupoid $G=H\ltimes Y$ is a second countable,
locally compact Hausdorff groupoid with a Haar system.  
Furthermore, the stabilizer subgroup of $H$ at $x$ is $H_x =
\{(q,0):q\in\Q\}$ for all $x\in Y$.  Since 
\(
(q,0)(r,0) = (q + r2^{0},0) = (q+r,0),
\)
the stabilizers are abelian, and since the stabilizers are also
constant, they must vary continuously in both $H$ and $G$.
It will be important for us to observe that $S$ is isomorphic to $\Q_D
\times Y$ via the map \(((q,0),x)\mapsto (q,x)\).  Finally,
$\{x_n\}_{n=0}^\infty$ forms a set of representatives for the orbit
space and it is not difficult to show that $Y/G$ is actually
homeomorphic to $\{x_n\}_{n=0}^\infty$ and is therefore Hausdorff.

To show that $\widehat{S}/G$ is not Hausdorff we must first
compute the dual.  Since $S$ is isomorphic to $\Q_D\times Y$ we can
identify $\widehat{S}$ with $\widehat{\Q_D}\times Y$.  While
$\widehat{\Q_D}$ is fairly mysterious we do know that since
\(
\hat{r}(s) = e^{ i r s}
\)
is a character on $\R$ for all $r\in \R$ it must also be a character
on $\Q_D$.  Now suppose $((q,n),x)\in G$ and
$(\hat{r},-n\cdot x)\in \Q_D\times Y$.  We have 
\begin{align*}
((q,n),x)\cdot (\hat{r},-n\cdot x)(p,x) &=
(\hat{r},-n\cdot x)(((q,n),x)\inv ((p,0),x)((q,n),x)) \\
&=(\hat{r},-n\cdot x)((-2^{-n}q,-n)(p,0)(q,n),-n\cdot x) \\
&=(\hat{r},-n\cdot x)((2^{-n}p,0),-n\cdot x) \\
&= e^{irp2^{-n}} = (\widehat{2^{-n}r},x)(p,x).
\end{align*}
Or, more succinctly, 
\begin{equation}
\label{eq:1}
((q,n),x) \cdot (\hat{r},-n\cdot x) = (\widehat{2^{-n}r},x).
\end{equation}
Next let \( \gamma_n = ((0,2n+1),(2^{-2n-1},0,0))\)
for all $n$. Using the inverse of \eqref{eq:3} we have 
\[
r(\gamma_n) = (2^{-2n-1},0,0)\quad\text{and}\quad
s(\gamma_n) = (2^{-2n},0,0).
\]
If we set \(\chi_n = (\hat{1},(2^{-2n},0,0))\) then clearly
 $\chi_n \to \chi = (\hat{1},(0,0,0))$.  Using 
\eqref{eq:1} we compute
\(
\gamma_n\cdot\chi_n = (\widehat{2^{-(2n+1)}},(2^{-2n-1},0,0)).
\)
A quick calculation shows that $\gamma_n\cdot\chi_n \to \omega =
(\hat{0},(0,0,0))$.  Hence $[\chi_n] \to [\chi]$ and $[\chi_n]\to
[\omega]$.  Since the action of $G$ is trivial on
fixed fibers this implies that $\widehat{S}/G$, and hence
$C^*(G)\sidehat$, is not Hausdorff.  
\end{example}

Even though our conjecture fails, 
we still know that if $G$ has continuously varying
abelian stabilizers and $G\unit/G$ is Hausdorff then
$C^*(G)\sidehat\cong \widehat{S}/G$.  What we need is an additional
hypothesis which, when taken in
conjunction with $G\unit/G$ being Hausdorff, will imply that
$\widehat{S}/G$ is Hausdorff.  The appropriate condition is given
below and forms the main result of the paper. 

\begin{theorem}
\label{thm:groupoidresult}
Suppose $G$ is a second countable locally compact Hausdorff groupoid
with a Haar system and abelian stabilizers.  Then $C^*(G)$ has
Hausdorff spectrum if and only if the following conditions hold:
\begin{enumerate}
\item the stabilizers vary continuously, i.e. $u \mapsto S_u$ is
  continuous with respect to the Fell topology, 
\item the orbit space $G\unit /G$ is Hausdorff, and,
\item given sequences $\{\chi_i\}\subset \widehat{S}$ and $\{\gamma_i\}\subset
  G$ with $\chi_i \in \widehat{S}_{s(\gamma_i)}$,  
  if $\chi_i\to \chi$ and $\gamma_i\cdot \chi_i \to \omega$
  such that $\chi$ and $\omega$ are in the same fiber then $\chi = \omega$.  
\end{enumerate}
\end{theorem}

In essence the third condition prevents the kind of ``looping''
behavior we see in Example \ref{ex:3} and is enough to guarantee
that $\widehat{S}/G$ is Hausdorff.  

\begin{remark}
Even in the case of transformation groups Theorem
\ref{thm:groupoidresult} is in some ways stronger than Theorem
\ref{thm:transresult}.  The main advantage is that we only require the
stabilizer groups to be abelian, and not the whole group.
Furthermore, we also removed the regularity hypothesis.  
The price is that we have added a slightly technical
condition that, while not easy to say, is simple enough to check in
practice.  
\end{remark}

\begin{proof}
In the discussion following our conjecture at the beginning of the
section on page \pageref{conj-1}  we showed that if
$C^*(G)\sidehat$ is Hausdorff then conditions (a) and
(b) hold and that $\widehat{S}/G$ is Hausdorff.   Now suppose we have
$\chi_i \to \chi$ and $\gamma_i\cdot \chi_i \to \omega$ as in
condition (c).  Then $[\chi_i]\to [\chi]$ and
$[\chi_i]\to[\omega]$.  Since $\widehat{S}/G$ is Hausdorff this
implies $[\omega] = [\chi]$.  However, 
$\chi$ and $\omega$ live in the same fiber and the action of $G$ on
a single fixed fiber is free so that $\chi = \omega$. 

Now suppose conditions (a)-(c) are satisfied.  Then again following
the discussion on page \pageref{conj-1}, the first two
conditions imply that $C^*(G)\sidehat$ is homeomorphic to
$\widehat{S}/G$.  Now suppose $[\chi_i]\to [\chi]$ and $[\chi_i]\to
[\omega]$ in $\widehat{S}/G$.  Using the fact that the quotient map is
open we can pass to a subsequence, relabel, and choose new representatives
$\chi_i$ so that $\chi_i \to \chi$.  As before let $p:\widehat{S}\to
G\unit$ be the bundle map and let $\tilde{p}:\widehat{S}/G \to
G\unit/G$ be the natural factorization.  Define $u_i = p(\chi_i)$ and
$u = p(\chi)$ and observe that $[u_i] \to [u]$.  Furthermore if
$p(\omega) = v$ then $[u_i] \to [v]$ as well.  Since
$G\unit/G$ is Hausdorff we have $[u] = [v]$ and we may assume, without
loss of generality, that $u = v$.  Now pass to a subsequence again,
relabel, and find $\gamma_i\in G$ 
such that $\gamma_i \cdot \chi_i \to \omega$.   
These sequences satisfy the hypothesis of (c) so $\omega = \chi$.  
It follows $[\omega] = [\chi]$ and that
$\widehat{S}/G$, and hence $C^*(G)\sidehat$, is Hausdorff.  
\end{proof}

It should be noted that there are a variety of situations in which
condition (c) is guaranteed to hold. 

\begin{prop}
Let $G$ be a second countable, locally compact Hausdorff groupoid with
continuously varying abelian stabilizers.  Then
condition (c) of Theorem \ref{thm:groupoidresult} automatically holds
if $G$ satisfies any of the following:
\begin{enumerate}
\item $G=H\ltimes X$ is an abelian transformation groupoid,  
\item $G$ is principal,
\item $G$ is proper,
\item $G$ is Cartan, or
\item $G$ is transitive. 
\end{enumerate}
\end{prop}

\begin{proof}
Let $\chi_i \to \chi$ and $\gamma_i\cdot\chi_i \to \omega$ be as in
condition (c).  Set $u_i = s(\gamma_i)$, $v_i = r(\gamma_i)$,  $u =
p(\chi) = p(\omega)$ and observe that $u_i \to u$ and $v_i \to u$.
Now suppose $G = H\ltimes X$ where $H$ is abelian.  Then we must have
$\gamma_i = (t_i, v_i)$ with $u_i = t_i\inv\cdot v_i$.  Given $s$ in
the stabilizer subgroup $H_u$ 
we can use the fact that the stabilizers vary continuously to
pass to a subsequence, relabel, and find $s_i\in H_{u_i}$ such
that $s_i \to s$ in $H$.  
Consequently $(s_i,u_i) \to (s,u)$ and by Proposition \ref{prop:3}
\(
\chi_i(s_i,u_i) \to \chi(s,u).
\)
On the other hand, since the group is
abelian, we also have $s_i \in H_{v_i} = H_{t_i \cdot u_i}$ for all
$i$.  It follows that $(s_i,v_i)\to (s,u)$ in $S$ and therefore 
\[
(t_i,v_i)\cdot \chi_i(s_i,v_i) =
 \chi_i(t_i\inv s_i t_i, u_i) = \chi_i(s_i,u_i) 
\to \omega(s,u).
\]
Hence $\chi = \omega$ and condition (c) automatically holds for
abelian transformation groups.  

Moving on, condition (c) trivially holds if $G$ is principal.  
For the next two conditions observe the following.  Suppose we can
pass to a subsequence, relabel, and find $\gamma \in G$ such that
$\gamma_i \to \gamma$.  It follows that $\gamma_i \cdot \chi_i \to
\gamma \cdot \chi$ and therefore $\gamma\cdot \chi = \omega$.  
However, the range and source maps are continuous so
we must have $r(\gamma) = s(\gamma) = u$ and hence $\gamma\in S_u$.
The fibres of $S$ are abelian so that by \eqref{eq:2} $\omega =
\gamma\cdot \chi = \chi$.  Thus it will suffice to show that we can
prove $\gamma_i$ has a convergent subsequence.  However, if $G$ is either
proper or Cartan then this follows almost by definition.  

Finally, suppose $G$ is transitive.  
Since $G$ is also second countable \cite[Theorem
2.2]{groupoidequiv} implies that the map $\gamma \mapsto
(r(\gamma),s(\gamma))$ is open.  Thus we can pass to a subsequence,
relabel, and find $\eta_i \in G$ such that $r(\eta_i) = v_i$,
$s(\eta_i) = u_i$ and $\eta_i \to u$.  Observe that $\eta_i \inv
\gamma_i \in S_{u_i}$ for all $i$ so that
\(
\gamma_i \cdot \chi_i = \eta_i \cdot (\eta_i\inv \gamma_i \cdot
\chi_i) = \eta_i \cdot \chi_i.
\)
Thus $\gamma_i \cdot \chi_i = \eta_i \cdot \chi_i \to
u\cdot \chi = \chi$.  It follows that $\chi = \omega$ and condition
(c) holds in this case as well. 
\end{proof}

\section{Examples and Duality}
\label{sec:hausd-spectr-dual}

In this section we would like to begin by applying Theorem
\ref{thm:groupoidresult} to several examples.  

\begin{example}
Let $H = SO(3,\R)$, $X = \R^3\setminus\{(0,0,0)\}$ and let $H$ act on
$X$ by rotation.  It is clear that $H$ is not abelian, and therefore
we cannot apply Theorem \ref{thm:transresult}.  However, it does have
abelian stabilizer subgroups.  
Given a vector $v\in X$ it's easy to see that $S_v$ is
the set of rotations about the line described by $v$.  
In particular, this is isomorphic 
to the circle group and is therefore
abelian.  What is more, some computations show that the stabilizers
vary continuously and that the stabilizer
subgroupoid $S$ is homeomorphic to $X\times \mathbb{T}$.  This in turn
implies that the dual groupoid is homeomorphic to $X\times
\mathbb{Z}$.  Now suppose $(U_i\inv v_i,\chi_i)\to (v,\chi)$ and
$(U_i,v_i)\cdot (U_i\inv v_i,\chi_i)\to (v,\omega)$ as in condition
(c).  Given $\theta_i\to \theta$ in $\T$ we have from
Proposition \ref{prop:3} that 
\[
(U_i\inv v_i,\chi_i)(U_i\inv v_i,\theta_i) = \chi_i(\theta_i) \to 
(v,\chi)(v,\theta) = \chi(\theta).
\]
Using the fact that conjugating rotation about an axis $w$ by $V\in H$
gives us the corresponding rotation about $Vw$,  we also have 
\begin{align*}
(U_i, v_i)\cdot(U_i\inv v_i, \chi_i)(v_i,\theta_i) 
&=(U_i\inv v_i,\chi_i)(U_i\inv v_i,\theta_i) 
=\chi_i(\theta_i) \to 
(v,\omega)(v,\theta) = \omega(\theta).
\end{align*}
It follows that $\chi = \omega$ and condition (c) of 
Theorem \ref{thm:groupoidresult} holds. 
Finally, the orbit space $X/H$ is homeomorphic to
the open half-line and is therefore Hausdorff.  Thus we can conclude
that $C^*(H\ltimes X)$ has Hausdorff spectrum.  In fact 
\cite[Theorem 3.5]{specpaper} shows that
$C^*(H\ltimes X)$ is homeomorphic to $\widehat{S}/H\ltimes X = (0,\infty)\times
\mathbb{Z}$.  
\end{example}

\begin{example}
Let $E$ be a row finite directed graph with no sources. 
Recall that we can build the graph groupoid $G$ as in \cite{graphgroupoid}.  
Elements of $G$ are triples $(x,n,y)$ where $x$
and $y$ are infinite paths which are shift equivalent with lag $n$,
and elements of $G\unit$ are infinite paths.\footnote{We will be using the
  Raeburn convention for path composition \cite{cbmsgraph}.}
Furthermore, the groupoid $C^*$-algebra $C^*(G)$ is isomorphic to the graph
$C^*$-algebra.  Let us consider the conditions of Theorem
\ref{thm:groupoidresult}.  First, the stabilizers are all subgroups of
$\Z$ and hence abelian.  Furthermore, the groupoid $G$ will have nontrivial
stabilizers if and only if there exists an infinite path which is
shift equivalent to itself.  In other words, if and only if there is a
cycle.  Suppose a cycle on the graph has an entry.  Let $x$ be the
path created by following the cycle an infinite number of times. 
  For each $i\in
\N$ let $x_i$ be the path which, at its head, follows the cycle $i$
times and then has a non-cyclic tail leading off from the entry.   Because $x_i$
eventually agrees with $x$ on any finite segment we have $x_i \to x$.
However none of the $x_i$ are cycles so that $S_{x_i}$ is
trivial for all $i$.  On the other hand $S_{x} \cong n\mathbb{Z}$
where $n$ is the length of the cycle.  Thus the stabilizers do not
vary continuously.  This shows that in order for the stabilizers to
vary continuously no cycles in the graph can have entries.  A similar
argument shows that the converse holds as well.  

For the second condition we require that the orbit space $G\unit/G$ be
Hausdorff.  In this case the orbit space is the space of shift
equivalence classes.  Recall that the basic open sets in $G\unit$ are
the cylinder sets $V_a$.  More specifically, $a$ is a finite path and
$V_a$ is the set of all infinite paths which are initially equal to
$a$.   Given $[x]\in G\unit/G$ we will have $x\in G\cdot V_a$ if and only
if $x$ is shift equivalent to a path with initial segment $a$.  This
is equivalent to there being a path from any vertex on $x$ to the
source of $a$.  Conversely, $y\not\in G\cdot V_a$ if and only if there is
no path from any vertex on $y$ to the source of $a$.  Using these
facts it follows from a brief argument that $G\unit/G$ will be
Hausdorff if and only if given non-shift equivalent paths $x$ and $y$
there exists vertices $u$ and $v$ such that there is a path from a
vertex on $x$ to $u$, a path from a vertex on $y$ to $v$, and there is
no vertex $w$ which has a path to both $u$ and $v$.  

Finally, for the third condition we observe that given $(y,n,x)\in G$,
$(y,m,y)\in S$ and $\chi \in \widehat{S}_x$ we have 
\begin{equation}
\label{eq:5}
(y,n,x)\cdot \chi(y,m,y) = \chi((x,-n,y) (y,m,y) (y,n,x)) =
\chi(x,m,x).
\end{equation}
Now suppose $\chi_i\to \chi$ and $(y_i,n_i,x_i)\cdot \chi_i\to \omega$
in $\widehat{S}$ with $\chi,\omega \in \widehat{S}_x$. 
Notice that this implies that we must have $x_i\to x$
and $y_i\to x$ in $G\unit$.  Let $(x,n,x)\in S_x$.  Then
$(x_i,n,x_i) \to (x,n,x)$ and by Proposition \ref{prop:3}
\(
\chi_i(x_i,n,x_i) \to \chi(x,n,x).
\)
On the other hand we also know $(y_i,n,y_i) \to (x,n,x)$ so that,
using \eqref{eq:5} and Proposition \ref{prop:3},
\[
(y_i,n_i,x_i)\cdot \chi_i(y_i,n,y_i) = \chi_i(x_i,n,x_i) \to
\omega(x,n,x).
\]
This implies that
$\chi(x,n,x)= \omega(x,n,x)$.  Hence $\chi = \omega$ and condition (c)
is automatically satisfied.  Put together this shows that the graph
groupoid algebra, and therefore the graph algebra, will have Hausdorff
spectrum if and only if 
\begin{itemize}
\item no cycle has an entry and,
\item given non-shift
equivalent paths $x$ and $y$ we can find vertices $u$ and $v$ such
that there is a path from a vertex on $x$ to $u$, a path from a vertex
on $y$ to $v$, and there is no vertex $w$ which has a path to both
$u$ and $v$.  
\end{itemize}
\end{example}

One annoyance of Theorem \ref{thm:groupoidresult} is that condition
(c) requires us to deal with the dual stabilizer groupoid.  Using the
same technique as the proof of Theorem \ref{thm:groupoidresult} one
can show that if $G\unit/G$ is Hausdorff and if condition (c) holds
for sequences in $S$ (not $\widehat{S}$) then $S/G$ is Hausdorff.
This raises the question of whether $\widehat{S}/G$ is Hausdorff if and
only if $S/G$ is Hausdorff, which is interesting in its own right.
Similar to the previous section we find that this question can be
answered in the affirmative in the $T_0$ and $T_1$ cases.  More
specifically, using the
topological argument given in Proposition \ref{prop:1}, one can prove
the following result. 

\begin{prop}
\label{prop:2}
Let $G$ be a second countable, locally compact Hausdorff groupoid with
continuously varying abelian stabilizers.  Then either $G\unit/G$,
$S/G$, and $\widehat{S}/G$ are all $T_1$ (resp. $T_0$) or none of them is
$T_1$ (resp. $T_0$).  
\end{prop}

Unfortunately, again similar to the previous section, 
this proposition doesn't extend to the Hausdorff case
either, as we demonstrate below.  This example also shows
that is is not enough to verify (c) on $S$ and that working with the dual
is necessary.  

\begin{example}
Let $H$, $Y$ and $G$ be as in Example \ref{ex:3}.  Recall
that we have already shown that in this case $\widehat{S}/G$
is not Hausdorff.  The computations from Example \ref{ex:3} also show
that condition (c) does not hold on $\widehat{S}$.  Now
we will show that $S/G$ is Hausdorff and that $S$ does
satisfy condition (c).  First, given
$((q,n),y)\in G$ and $(r,x)\in S$ a computation similar to
the one preceding \eqref{eq:1} shows that
\begin{equation}
\label{eq:151}
((q,n),y)\cdot (r,x) = (r2^n,y).
\end{equation}
Suppose $[s_i]\rightarrow [s]$ and $[s_i]\rightarrow [t]$ in $S/G$.  
Since $Y/G$ is Hausdorff we can follow the same argument given
in Theorem \ref{thm:groupoidresult} to pass to subsequences, choose
new representatives, and find $\gamma_i\in G$ so that 
$s_i\rightarrow s$ and $\gamma_i\cdot s_i
\rightarrow t$ where $s,t\in S_u$.  In particular this implies $s =
(r,u)$ and $t = (q,u)$ for $r,q\in\Q$.  Suppose $s_i = (r_i,x_i)$ and $\gamma_i =
((p_i,n_i),y_i)$.  Then it follows from \eqref{eq:151} that
$\gamma_i\cdot s_i = (r_i 2^{n_i},y_i)$.  Hence $r_i\rightarrow r$ and
$r_i2^{n_i}\rightarrow q$.  However, we gave $\Q_D$ the discrete
topology so that, eventually,
\(
q = 2^{n_i}r_i = 2^{n_i}r.
\)
Now, if either $r = 0$ or $q=0$ then $s=t$.  If $r,q\ne 0$ we know
that eventually $n_i = n =\log_2(q/r)$.  We may as well pass to a
subnet and assume this is always true.  But then $n_i \cdot x_i
\rightarrow n\cdot x$.  However, we also have $n_i\cdot x_i =
\gamma_i\cdot x_i = y_i \rightarrow x$.  Thus $n\cdot x = x$.  But the
action of $\Z$ is free which implies $n=0$.  Thus $\log_2(q/r)
= 0$ and $q=r$.  It follows that $s=t$ and that $S/G$ is Hausdorff.
What is more, the above argument also shows that condition 
(c) holds for sequences in $S$.  
\end{example}


\providecommand{\bysame}{\leavevmode\hbox to3em{\hrulefill}\thinspace}
\providecommand{\href}[2]{#2}

\end{document}